\documentclass[	11pt,twoside,a4paper]{article}
\usepackage[cp1251]{inputenc}
\usepackage{amsmath,amssymb,amsthm}

\usepackage{hyperref}

\newcommand{\V}{\mathrm V}
\newcommand{\E}{\mathrm E}
\newcommand{\W}{\mathrm W}
\renewcommand{\d}{\partial}

\newcommand{\R}{\mathbb R}
\newcommand{\Z}{\mathbb Z}

\newcommand{\hg}{\widehat{g}}
\newcommand{\s}{\mathfrak s}
\newcommand{\id}{\operatorname{id}}
\newcommand{\type}{\operatorname{type}}

\newcommand{\sub}{\subset}

\renewcommand{\phi}{\varphi}
\newcommand{\e}{\varepsilon}

\renewcommand{\a}{\alpha}

\newcommand{\g}{\gamma}

\newcommand{\lasdim}{\mathrm{\ell\hbox{-}asdim}\,}
\newcommand{\asdim}{\mathrm{asdim}\,}

\theoremstyle{theorem}

\newtheorem{theorem}{Theorem}[section]

\newtheorem{corollary}[theorem]{Corollary}
\newtheorem{lemma}[theorem]{Lemma}

\theoremstyle{definition}
\newtheorem{definition}{Definition}

\theoremstyle{remark}
\newtheorem{remark}{Remark}

 \newtheoremstyle{break}
   {9pt}
   {9pt}
   {\itshape}
   {}
   {\bfseries}
   {.}
   {\newline}
   {}

 \theoremstyle{break}

\begin{document}

\title{
Four-dimensional graph-manifolds with fundamental groups
quasi-isometric to fundamental groups of orthogonal graph-manifolds}
\author{
Aleksandr Smirnov\footnote{This work is supported by the Program of the 
Presidium of the Russian Academy of Sciences 01 
"Fundamental Mathematics, and its Applications" under
grant PRAS-18-01, and by RFBR Grant 17-01-00128a }
}
\date{}
\maketitle

\begin{abstract} 
We introduce a topological invariant, {\it a type} of a graph-manifold, which 
takes 
natural values.
For a 4-dimensional graph-manifold, whose type does not exceed two,
it is proved that its universal cover is bi-Lipschitz equivalent
to a universal cover of an orthogonal graph-manifold (for
any Riemannian metrics on graph-manifolds).
\end{abstract}

\section{Introduction}\label{sec:intr}

The main result of this paper (see Theorem~\ref{thm:type2}) is a bi-Lipschitz
equivalence of universal covers for some
classes of 4-dimensional graph-manifolds.
The motivation for this result is the problem of finding an asymptotic 
invariants of graph-manifolds, in particular, 
asymptotic 
($\asdim \pi_1(M)$) 
and linearly-controlled asymptotic 
($\lasdim \pi_1(M)$) 
dimensions of their 
fundamental groups. 
Theorem~\ref{thm:type2} allows us to reduce the finding of these dimensions for
a wide class of graph-manifolds to results from~\cite{Smir1}.

In 3-dimensional case,
$\dim M = 3$,
the problem of finding asymptotic dimensions
was solved in~\cite{HS}.
For the case, when 
$\dim M\geq 4$,  
asymptotic dimensions 
$\asdim \pi_1(M)$ 
and 
$\lasdim\pi_1(M)$ 
was found only for graph-manifolds of a special type, 
called orthogonal graph-manifolds.
Namely, in~\cite{Smir1} for orthogonal
graph-manifolds it is proved that
$$
\asdim \pi_1(M)=\lasdim \pi_1(M)=\dim M.
$$
The definition of these invariants can be found, for example, 
in~\cite{BS},~\cite{Gro},~\cite{Smir1}.

Orthogonal graph-manifolds in the 3-dimensional case are so-called flip
graph-manifolds, for which gluings between blocks
are especially simple.
According to~\cite{KL}, the fundamental group of any closed 3-dimensional
graph-manifold is quasi-isometric to the fundamental group of a flip 
graph-manifold, therefore, fundamental groups of any closed 
3-dimensional graph-manifolds are pairwise quasi-isometric.
In higher dimensions this is not true.
In this paper we introduce a topological invariant,
$\type M$,
of the graph-manifold
$M$,
which takes natural values.
In any dimension greater than~3 it is not difficult to construct a 
graph-manifold of any type.
However, for 4-dimensional orthogonal graph-manifolds, 
the type is always does not exceed~2.
The main result of this paper is as follows.

\begin{theorem}\label{thm:type2}
If the type of a 4-dimensional graph-manifold
$M$ does not exceed two,
$\type M \leq 2$,
then its universal cover is bi-Lipschitz equivalent to 
the universal cover of some orthogonal graph-manifold (for
any Riemannian metrics on graph-manifolds).
\end{theorem}

\begin{corollary}\label{cor:intrees}
For the fundamental group of any 4-dimensional graph-manifold
$M$, with 
$\type M \leq 2$, 
there is a quasi-isometric embedding into the product of 4 metric trees, and, 
consequently, 
$\asdim \pi_1(M) = \lasdim \pi_1(M) = 4$,
where
$\asdim$ 
and
$\lasdim$ are
asymptotic and linearly-controlled
asymptotic dimensions.
\end{corollary}

This corollary gives us a simple and easily verifiable sufficient condition that
allows us to calculate
$\asdim \pi_1(M)$ 
and 
$\lasdim \pi_1(M)$.
 
In addition, the class
$\mathcal{GM}_2$
of graph-manifolds with
$\type M \leq 2$
is much wider (see Section~\ref{sec:exmpl}) than the class of orthogonal
graph-manifolds.
Moreover, it is highly doubtful that any graph-manifold in the class
$\mathcal{GM}_2$ has a finite cover by an orthogonal graph-manifold. 
 
As an important additional result we give a criterion of orthogonality of
4-dimensional graph-manifolds, which blocks have type~2, 
see Theorem~\ref{thm:crit}.
As a consequence, we obtain a wide class of 4-dimensional graph-manifolds,  
which type is equal to~2, but which are not orthogonal (see 
Corollary~\ref{cor:notort}).

The proof of Theorem~\ref{thm:type2} consists of two steps.  
An important role is played by the intersection number, 
and the secondary intersection number that 
in the general case can be arbitrary positive integers (see 
section~\ref{subsec:ind1and2}).
For orthogonal graph-manifolds such numbers are equal to~1.
In the first step, in Section~\ref{sec:virtual},
passing to a finite cover of the graph-manifold
$M$,
we are building
4-dimensional graph-manifold
$N$, with 
intersection numbers, and secondary
intersection numbers equal to~1.
Since, the universal cover of graph-manifolds
$M$
and
$N$
coincide, then we only need to prove
Theorem~\ref{thm:type2} for the graph-manifold
$N$.
 
In the second step, we "re-glue" the graph-manifold
$N$
to an orthogonal graph-manifold without changing the bi-Lipschitz type of 
its universal cover. 
The procedure of the re-gluing is described in Section~\ref{sec:maintheorem}.
It is a generalization of the procedure used in~\cite{KL} for
3-dimensional graph-manifolds.

Corollary~\ref{cor:intrees} follows from the result from~\cite{Smir1},
claiming that the fundamental group of any $n$-dimensional orthogonal 
graph-manifold
$M$
can be quasi-isometrically embedded into the product of  
$n$
metric trees, and, consequently,
$$
\ asdim \ pi_1 (M) = \ lasdim \ pi_1 (M) = n.
$$

\section{Preliminaries}\label{sec:pre}

\subsection{Graph-manifolds}\label{subsec:grm}

\begin{definition}\label{def:gm}
Let 
$n\geq 3$. 
{\em A higher-dimensional graph-manifold} is a closed, oriented, 
$n$-dimensional 
manifold
$M$ 
that is glued from a finite number of blocks 
$M_v$, 
$M = \bigcup_{v\in V}M_v$. 
It should satisfy the following conditions (1)--(3).
\begin{itemize}
 \item[(1)] Each block $M_v$ is a trivial $(n-2)$-dimensional tori $T^{n-2}$ 
 bundle over a compact, oriented surface $\Phi_v$ 
 with boundary (the surface should be different from the disk, and the annulus);
 \item[(2)] the manifold $M$ is glued from blocks  $M_v$, $v \in V$, 
 by diffeomorphisms between the boundary components (we does not exclude the 
case 
 of gluing the boundary components of a single block);
 \item[(3)] gluing diffeomorphisms do not identify the homotopy 
 classes of fiber tori.
\end{itemize}
\end{definition}

These graph-manifolds for $n\geq 4$ were introduced in~\cite{BK}.

To each graph-manifold 
$M$
the graph
$G$, 
dual to its block decomposition, corresponds. 
Thus, the set of graph-manifold blocks coincides with the set of vertices
$\V$ 
of the graph 
$G$,
and the set of pairs of gluing blocks coincides with the set of edges 
$\E$ 
of the graph 
$G$.
The set of all directed edges of the graph 
$G$ 
we denote by 
$\W$.

Orthogonal graph-manifolds defined in~\cite{Smir1} 
are distinguished in the class of graph-manifolds only by the condition 
for gluing diffeomorphisms.
They are obtained as follows.

For each vertex  
$v\in \V$
we fix a trivialization of the fibration  
$M_v\to \Phi_v$,
that is,   
we represent the block    
$M_v=\Phi_v\times S^1\times\ldots\times S^1$ 
as the product, 
where  
$S^1$ 
occurs 
$n-2$ 
times.
Thus, for each edge   
$w=\{vv'\}$,
adjacent to the vertex  
$v$,
we have a trivialization of a boundary torus 
$T_{w}=S^1\times S^1\times\ldots\times S^1$,
$(n-1)$ 
times,  
of the block  
$M_v$, 
that corresponds to the edge  
$w$.
  
In the same way, for each edge 
$-w$,
going in the opposite direction,  
we have a trivialization of the boundary torus   
$T_{-w}=S^1\times S^1\times\ldots\times S^1$ 
of the block  
$M_{v'}$.

We fix an order on the set of all factors of the trivialization, and define 
a gluing diffeomorphism of the tori 
$T_w$ 
and  
$T_{-w}$ 
by some permutation  
$\s_w$ 
of factors of the trivialization, 
that does not identify the boundary components  
$\Phi_v$ 
and 
$\Phi_{v'}$.

Note that this map is a well-defined gluing,  
as permutations $\s_w$, and $\s_{-w}$
are selected to be mutually inverse. 
Also, the map $\eta_w$ does not identify the homotopy
classes of fiber tori.

In this case, for edges,   
$w$ 
and
$-w$, 
going in the opposite directions 
permutations 
$\s_w$ 
and 
$\s_{-w}$
are selected to be mutually inverse, i.e. 
$\s_{-w}\circ \s_{w}=\id$.

In other words, a graph-manifold is orthogonal iff there exists such
trivialization of all blocks that
the gluing maps are defined by permutations of the factors as described above. 
The disadvantage of this definition is that 
you can not verify whether a given graph-manifold orthogonal or not with 
this definition. 
It depends on the choice of the trivialization of the blocks which is not 
unique. 
For another choice of trivializations of gluing blocks the graph-manifold 
can cease to be orthogonal.
In Section~\ref{sec:exmpl} we present a criterion of orthogonality 
of some class of 4-dimensional graph-manifolds.
This criterion does not depend on the choice of trivializations.

\subsection{W-structure}\label{subsec:wstr}

The main tool for working with graph-manifolds is the so-called
$W$-structure, first described in the 3-dimensional case in works 
of Waldhausen~\cite{W1},~\cite{W2}.
For the 
$n$-dimensional
case the definition of 
$W$-structure
is given in~\cite{BK}.
For the convenience of the reader, we give these definitions here.

Let  
$G$ be the graph of a graph-manifold  
$M$.
For a vertex
$v\in \V$ 
by
$\d v$ 
we define the set of all directed edges, 
adjacent to  
$v$.

Note that, to each directed edge  
$w\in \W$ 
the homology group  
$L_w = H_1(T_w;\Z)\simeq \Z^{n-1}$\label{resh}
of the gluing torus
$T_w$ 
corresponds. 
Furthermore, 
for each vertex 
$v\in V$ 
the homology group 
$F_v = H_1(T_v;\Z) \simeq \Z^{n-2}$ 
of the fiber 
$T_v$ 
of the block
$M_v$
corresponds.

Moreover,
if 
$w\in \partial v$,
then there exists an 
embedding of 
the group 
$F_v$ 
in the group
$L_w$ 
as a maximum subgroup~$F_w$.

We call the group 
$F_v \simeq F_w$ 
{\it a fiber group}.
For each orientation of the graph-manifold
$M$
there are fixed corresponding orientations of each 
block of 
$M$, and fixed corresponding orientations of 
groups
$L_w$,
$w\in \W$.
Along with it, orientations of groups 
$L_w$
and
$L_{-w}$
are opposite.

A gluing of blocks is described by an isomorphism
$\hg_w\colon L_{-w}\to L_w$,
that satisfy the conditions
\begin{align}
	\hg_{-w}&=\hg^{-1}_w;\label{w1}\\ 
	\hg_w(F_{-w})&\neq F_w. \label{w2} 
\end{align}

The choice of a trivialization of each block
$M_v$, 
as well as a trivialization of the fiber, 
fixes for each edge 
$w\in\d v$ 
a basis of the group  
$L_w$
(up to the choice of its elements signs)
so that its corresponding subset of elements forms 
a basis of the group
$F_w$.

Such bases are called \textit{selected.}

Let us describe a set of selected bases groups
$L_w, w\in \W$ in terms of their transformation  groups. 
Let  
$f_v = (f^1_v,\ldots,f^{n-2}_v)$ 
be a basis of the group 
$F_v$.

We choose a basis
$(z_w, f_w)$ 
of the group 
$L_w$
so that
$f_w=f_v$, 
and there exists a trivialization $M_v =\Phi_v\times T^{n-2}$.
such that the set 
$\{z_w\mid w\in \partial v\}$ 
corresponds to the boundary of the surface
$\Phi_v$. 

In this case, the basis 
$f_v$
defines some orientation of the fiber
$F_v$, 
and the basis 
$(z_w,f_w)$
defines some orientation of the group 
$L_w$.

The group of transformations of these bases consists of matrices of the form
$$h_w=\left(\begin{array}{cc}
  \varepsilon_v & 0 \\
   n_w & \sigma_v
     \end{array}\right),
$$
where
$\varepsilon_v = \pm 1$,
$n_w \in \Z^{n-2}$,
$\sigma_v \in GL(n-2,\Z)$,
and acts on bases:
$$
(z_w,f_w)\cdot h_w = (z_w\cdot \varepsilon_w +f_w \cdot n_w, f_w\cdot \sigma_v).
$$
We require that for each vertex
$v\in V$ 
the following conditions are fulfilled:

\begin{align}
	 \varepsilon_v\cdot \det \sigma_v &= 1; \label{w3}\\ 
	 \sum\limits_{w\in \partial v} n_w &= 0.\label{w4}
\end{align}
It is easy to see that the set
$\mathcal H$ 
of matrices of the form 
$$
h=\bigoplus\limits_{w\in \W} h_w, 
$$
satisfying the conditions~\eqref{w3},~\eqref{w4}, is a group.
The condition~\eqref{w3} means that each basis
$(z_w, f_w)$ 
agreed with the fixed orientation of the group
$L_w$,
and the condition~\eqref{w4} means that these bases correspond to
some trivialization of the block 
$M_v$.

$W$-\textit{structure}, associated with a graph-manifold
$M$,
$M$,
is a collection of groups  
$\{L_w\mid w\in \W\}$,
satisfying the conditions~\eqref{w1},~\eqref{w2}, and the set of 
their bases of the form 
$\Theta = (z, f)\cdot \mathcal H$,
where  
$(z,f)$ is the set of bases mentioned above  
and 
$\det g^{z,f}_w = -1$ 
for each directed edge 
$w\in \W$.

The last condition means that the isomorphism 
$\hg_w:L_{-w}\to L_w$
reverses the orientation. The elements
$(z,f)\in \Theta$
are called \textit{a Waldhausen basis}.

For a fixed Waldhausen basis 
$(z, f)$
the gluing isomorphism is described by the matrix
$$g_w=g_w^{z,f}=\left(\begin{array}{cc}
  a_w & b_w \\
   c_w & d_w
     \end{array}\right),
$$
where 

\begin{align}
 (z_{-w}, f_{-w}) = (z_w, f_w)\cdot g_w\label{w5}
\end{align}
(it is assumed that the groups
$L_{-w}$
and
$L_w$
are identified by the isomorphism
$\hg_w$).
Here  
$a_w \in \Z$.  
The row  
$b_w$, 
and the a column,  
$c_w$ 
consist of 
$n-2$ 
integers. 
The matrix 
$d_w$ 
is an integer matrix of size
$(n-2)\times (n-2)$.

\begin{remark}\label{rem:ort}
A graph-manifold 
$M$ 
is orthogonal iff 
on each block of the graph-manifold 
$M$ 
there is such a trivialization, 
that for each directed edge
$w\in\W$
its induced bases
$(z_w,f_w)$ 
and 
$(z_{-w},f_{-w})$ 
of the groups   
$L_w$ 
and 
$L_{-w}$ 
differ only by permutation of the elements, and, perhaps,
the statement of signs before vectors.
\end{remark}

\subsection{Fiber subspaces, and intersection of lattices}\label{subsec:fs}

In what follows, we will use subgroups of groups isomorphic to~$\Z^n$. 
In this case, the maximum subgroups will play an important role.
For brevity, we will call them {\it lattices}.

\textit{The lattice in a group} 
$G$ 
isomorphic to
$\Z^n$ 
is a maximum subgroup 
$H$ 
which is isomorphic to  
$\Z^k$ 
for some 
$k\leq n$. 
That is, such a subgroup,
for which
there is no another subgroup
$H'<G$ 
isomorphic to 
$\Z^k$,
such that  
$H < H'$. By
{\it the dimension} 
of the lattice we call the number
$k$.

\begin{remark}\label{rem:resh}
The intersection
$G_1\cap G_2$ 
of two lattices 
$G_1$ 
and 
$G_2$
is a lattice.
It follows from the fact that if 
$\gamma^m\in G_i$,
$m\neq 0$,
then 
$\gamma\in G_i$,
$i=1,2$.
\end{remark}

For each edge  
$|w|\in \E$ 
we denote the intersection of the lattices
$F_w$ 
and 
$F_{-w}$ 
by 
$P_{|w|}$.

Further, such a lattice in  
$L_{|w|}$ 
we call 
\textit{the intersection lattice} for the edge~$w$.

\begin{definition}
For any edges 
$w$,
$w'\in\d v$
the lattices 
$P_{|w|}$ 
and  
$P_{|w'|}$
are called 
{\em parallel}  
iff they are the same in the group
$F_v$.

In this case,
for brevity, we also say that the edges
$w$ and 
$w'$
{\em are parallel}. 
\end{definition}

\begin{definition}
Lattices 
$P_{|w|}$,
considered as subgroups of the group
$F_v$,
$w\in \d v$,
we call 
\textit{intersection lattices} 
for the vertex~$v$.  
\end{definition}

\subsection{Type of a block, and graph-manifolds}\label{subsec:type}

Let 
$M_v$ 
be a block of a graph-manifold
$M$
that corresponds to a vertex 
$v$.

\begin{definition}
{\em The type} of the vertex  
$v$ 
(or of the block  
$M_v$
) 
is the maximal number of pairwise non-parallel edges 
$w\in \d v$.

We denote the type of the vertex (of the block) by  
$\type{v}$ 
($\type{M_v}$). 
{\em The type of a graph-manifold} $M$ is the maximal type of its blocks  
$\type{M}:=\max\limits_{v\in \V}\type{v}$.
\end{definition}

\begin{remark}\label{rem:type}
The type of a block, and consequently the type of the graph-manifold, 
do not depend on the choice of the Waldhausen basis. It means that they  
are topological invariants of the graph-manifold.
\end{remark}

In this paper we consider only graph-manifolds of the dimension 4, and 
we are interested in blocks of the type 1 and 2.
For each block 
$M_v$ 
of the type 1 we denote the unique intersection lattice of the vertex 
$v$ 
by
$P^1_v$.
For each block 
$M_v$ 
of the type 2 we denote the corresponding intersection lattices by
$P^1_v$ 
and 
$P^2_v$.

Further in the paper, by a graph-manifold we mean the 4-dimensional
graph-manifold, unless otherwise stated.

\subsection{Intersection number 
and secondary intersection number}\label{subsec:ind1and2}

Recall the definitions of some invariants of  
$W$-structures, described in~\cite{BK}.

Using the condition~\eqref{w2}, we have that an integer string 
$b_w$ 
is non-zero.
It means that the greatest common divisor 
$i_w\geq 1$
of its elements is defined.

\begin{definition}\label{def:ind}
The number
$i_w$
is called the \textit{intersection number}
of the 
$W$-structure
on the edge
$w$.
\end{definition}

Geometrically
$i_w$
is a number of intersection components of 
$T_w\cap T_{-w} \sub T_{|w|}$.
 
As 
$i_w=i_{-w}$,
that is, the intersection number is independent of the edge direction. 
It means, that 
we can use the intersection number of the undirected edge 
$e=|w|$,
as 
$i_e=i_{w}=i_{-w}$.

Let  
$F_e$ 
be a smallest subgroup of the group  
$L_e$,
which contains 
$F_w$ 
and
$F_{-w}$,
$F_e=\langle F_w,\ F_{-w}\rangle$.
 
\begin{lemma}\label{lemm:ind1}
The subgroup index
$(L_e:F_e)$ 
is equal to the intersection number  
$i_e$ 
of the edge 
$e$.
\end{lemma}
\begin{proof}
Group
$F_e$
is generated by elements
$f_w$,
$f_{-w}$,
$F_e=\langle f_w,\ f_{-w} 
\rangle$,
while
$L_e=\langle z_w,\ f_w \rangle$.

It follows from the condition~(\ref{w5}) 
that elements
$b^1_w\cdot z_w$,
$b^2_w\cdot z_w$,
\ldots, 
$b^{n-2}_w\cdot z_w$ 
belong to the group 
$F_e$,
where 
$b_w=(b^1_w,\ldots, b^{n-2}_w)$.

Since the intersection number 
$i_e$ 
is equal to the greatest common divisor of numbers 
$b^1_w$,
$b^2_w$,
\ldots,  
$b^{n-2}_w$,
then   
$\a\cdot z_w\in F_e$,  
iff 
$\a$ 
is divisible by 
$i_e$.
Therefore
$(L_e:F_e)=i_e$.
\end{proof}

For each block
$M_v$ 
of the type 2, the intersection lattices 
$P^1_v, P^2_v\subset F_v$ 
generate a subgroup
$P_v\simeq \Z^2$,
$P_v=\langle P^1_v, P^2_v \rangle$ 
(the least subgroup in 
$F_v$,
containing
$P^1_v$ 
and 
$P^2_v$) 
in the group  
$F_v$.

\begin{definition}\label{def:slper}
We call the group  
$P_v$ 
\textit{the group of fiber intersections}  
$M_v$.
\end{definition}

The subgroup
$P_v$
is not necessarily maximum, so it may not be a lattice in
$F_v\simeq \Z^2$.

\begin{definition}\label{def:in2}
The index  
$j_v$ 
of the subgroup
$P_v$ 
in the group
$F_v$ 
is called \textit{secondary intersection number} 
at the vertex 
$v$.
\end{definition}

For each block  
$M_v$ 
of the type 1, 
we choose  
$P_v=F_v$ 
and
$j_v=1$.
(Since 
$P^1_v\simeq \Z$,
and
$P_v\simeq \Z^2$,
then
$P_v\neq P^1_v$.
)

\section{Unwinding of intersection numbers up to 1.}\label{sec:virtual}

In this section we prove that for any graph-manifold 
there is a finite sheet cover by a graph-manifold with 
all intersection numbers, as well as all the secondary intersection numbers 
equal to 1.

\begin{lemma}\label{lem:virtual2}
For any graph-manifold
$M$
there is a graph-manifold
$N$,
for which all intersection numbers are equal to~1, 
and also all secondary intersection numbers are equal to~1.
\end{lemma}

\begin{proof}
For each vertex  
$v\in \V$ 
we consider a cover
$r_v\colon T^2\to T^2$,
corresponding to a subgroup 
$P_v<F_v=\pi_1(T^2)$
of the fundamental group of the fiber torus 
(for a vertex of the type 1 such a cover is 
trivial). 
The degree of this cover is equal to the secondary 
intersection number $j_v$ 
at the vertex~$v$.

We consider a surface 
$\Phi_v$ 
and construct an orbifold
$\Phi'_v$ 
as follows.
For each edge  
$w\in \d v$ 
we glue a disk 
$D_w$ 
with a conical point with an angle 
$2\pi/j_{u}$,
where 
$u$
is the other end of the edge 
$w$, 
to the corresponding component of the boundary of the surface
$\Phi_v$. 

Since the surface with boundary
$\Phi_v$ 
is different from the disk, and the ring, 
then the orbifold 
$\Phi'_v$ 
is good, compact 2-dimensional orbifold without boundary, 
and therefore, 
(see~\cite[Theorem 2.5.]{Sc}) 
there are a closed surface
$\Psi'_v$ 
and a finite cover
$p'_v\colon \Psi'_v\to \Phi'_v$.

Let  
$n_v$
be a degree of this cover
We denote the product 
$\prod\limits_{v\in \V}{n_v}$ 
by~$N$,
and the product  
$\prod\limits_{v\in \V}{j_v}$~---  
by~$J$.

Cutting out from the surface
$\Psi'_v$ 
the preimage  
$(p')^{-1}_v(\bigcup\limits_{w\in \d v}{D_w})$ 
of the glued disks, 
we obtain a surface  
$\Psi_v$ 
with boundary, 
which covers the surface
$\Phi_v$ with boundary with finite degree. 
We denote the corresponding cover by  
$p_v\colon \Psi_v\to \Phi_v$.

Let 
$N_v=\Psi_v\times T^2$.
Such manifolds we also call blocks.
We define the cover  
$$
q_v\colon N_v\to M_v=\Phi_v\times T^2, 
$$ 
as the product of covers
$p_v\colon \Psi_v\to \Phi_v$ 
and 
$r_v\colon T^2\to T^2$.

Note that in the block  
$N_v$ 
the group  
$P_v$ 
plays the role of a fiber group.

Let 
$w$ 
be an edge from 
a vertex 
$v$ 
to a vertex
$u$,
$\g_w$ 
be a boundary component of the surface 
$\Phi_v$, 
corresponding to an edge
$w$.
On each component of the  preimage of the torus
$T_{w}=\g_w\times T^2$ 
the cover  
$q_v$ 
is a product of covers, and
is determined by 
subgroup
$A_{w}=\langle P_v,\ P_u\rangle=B_w\times P_v$
of the group 
$L_{w}$,
where  
$B_w$
is a subgroup 
of the group 
$\pi_1(\g_w)\simeq \Z$
with subgroup index 
$j_u$.

Thus, the group
$A_w$ 
has the subgroup index  
$(L_w:A_w)=j_u\cdot j_v$ 
in the group
$L_w$. 

Now we describe gluings of blocks.
Let 
$w$ 
be an edge from 
a vertex 
$v$ 
to a vertex
$u$. 
Let  
$T_v$
be a boundary component of the block 
$N_v$,
and let
$T_u$
be a boundary component of the block 
$N_u$. 
Let $T_u$, and $T_v$ covers the torus 
$T_{|w|}$.
A gluing  
$g'_w\colon T_v\to T_u$,
is defined by an isomorphism of groups 
$H_1(T_v;\Z)$ 
and 
$H_1(T_u;\Z)$.
  
For each of these groups we have 
$A_{w}$,
which is induced by covers
$q_v$ 
and
$q_u$ 
respectively. 
This determines the required gluing.
Since the lattices 
$P_u, P_{v}<A_{w}$ 
are different, then such a gluing
satisfies the definition~\ref{def:gm}~(3). 

Note that for the edge 
$w$,
from a vertex
$u$ 
to a vertex  
$v$ 
of the graph 
$G$,
we have 
$P_{|w|}=F_w\cap F_{-w}=P_u\cap P_v<L_w$.
 
Therefore, for the edge
$w'$,
that corresponds to gluing blocks
$N_u$ 
and  
$N_v$,
we have that its intersection lattice  
$P_{|w'|}=P_{|w|}$. 
Consequently, the group
$P_v$ 
plays the role of the fiber intersection group for the block
$N_v$.
It means, that the secondary intersection number of the block
$N_v$ 
is equal to  
$(P_v:P_v)=1$.

By Lemma~\ref{lemm:ind1} the intersection number on the edge
$w'$ 
is equal to the subgroup index  
$(A_w:\langle P_u, P_v \rangle)$ 
of the subgroup
$\langle P_u, P_v \rangle$ 
in the group
$A_w$,
i.e it is equal to~1.

For each vertex
$v\in \V$ 
we consider 
$N/n_v\cdot J/j_v$ 
copies of the block 
$N_v$.
  
For each edge
$w\in \W$ 
($e=|w|$), 
between 
$u,v\in \V$,
we have  
$N/n_u\cdot J/j_u$ 
copies of the block 
$N_u$ 
and  
$N/n_v\cdot J/j_v$ 
copies of the block 
$N_v$.
 
The block 
$N_u$ 
has 
$(n_u\cdot j_u)/(j_u\cdot j_v)=n_u/j_{v}$ 
boundary components, covers the torus 
$T_e$,
and the block  
$N_v$ 
has 
$n_v/j_{u}$ 
boundary components, covers the torus 
$T_e$.
 
Then all copies of the block  
$N_u$ 
have 
$N/n_u\cdot J/j_u\cdot n_u/j_v=(N\cdot J)/(j_u\cdot j_v)$ 
boundary components, covers the torus 
$T_e$.
 
All copies of the block  
$N_v$
have the same number of 
boundary components, covers the torus 
$T_e$.
 
We take some correspondence between these copies, and glue each boundary 
component  
of a copy of the block 
$M_u$ 
with the corresponding boundary component  
of a copy of the block 
$M_v$ 
by some gluing homeomorphism 
$g'_w$.
 
We obtain a graph-manifold  
$M'$, for which 
all secondary intersection numbers are equal to~1.
and which
$N\cdot J$~covers the graph-manifold~$M$.
\end{proof}
 
Applying Lemma~\ref{lem:virtual2} to the graph-manifold  
$M$, we obtain a graph-manifold 
$N$, 
for which all intersection numbers are equal to~1, 
and also all secondary intersection numbers are equal to~1.
Moreover, the fundamental groups 
$\pi_1(N)$ 
and 
$\pi_1(M)$ 
are quasi-isometric.

\section{The proof of Theorem~\ref{thm:type2}}\label{sec:maintheorem}

For the reader's convenience, we present here Lemma~2.4 from~\cite{KL}. 
This lemma plays an important role in the proof of Theorem~\ref{thm:type2}.  

\begin{lemma}\label{lemm:KL}
Let 
$S$ 
be a smooth compact manifold with strictly negative curvature, and 
totally-geodesic boundary.
Denote by  
$\widetilde{S}$
the universal cover of 
$S$.
Let  
$\alpha$
be a closed smooth 1-form on 
$\d S$. 
Denote by 
$\alpha'$
the pull-back of 
$\alpha$ to  
$\d \widetilde{S}$.

Then there exists a smooth Lipschitz function  
$h\colon \widetilde{S}\to \R$ 
satisfying 
$d h|_{\d \widetilde{S}}=\alpha'$.
\end{lemma}

Let 
$M$
be a 4-dimensional graph-manifold with $\type \leq 2$. 

Passing to the finite cover, we assume that 
all intersection numbers of 
$M$ 
are equal to~1, 
and also all secondary intersection numbers are equal to~1.

Since all secondary intersection numbers 
of 
$M$
are equal to~1, 
then we can
choose
such a Waldhausen basis 
$\{(z_w,f_w)\mid w\in \d v, v\in \V\}$,
that for each block 
$M_v$ 
of the type ~2, we have  
$f^1_v\in P^1_v$ 
and
$f^2_v\in P^2_v$.

Moreover for each block of the type~1 we can choose  
$f^1_v\in P^1_v$.

For each edge  
$w\in \W$,
from  
$v$ 
to  
$u$,
the gluing  
$\hg_{-w}$ 
of blocks  
$M_v$ 
and  
$M_u$ 
is given by bases  
$(z_w,f_w)$ 
and  
$(z_{-w},f_{-w})$ 
of the lattice  
$L_{|w|}$.
 
In other words, the matrix
$g_{-w}$ 
is obtained by decomposing the basis 
$(z_w,f_w)$ 
on the basis
$(z_{-w},f_{-w})$.
 
We can assume that
$P^1_u=P_w=P^1_v$,
where 
$P_w=F_w\cap F_{-w}$.
Then, in this notation,
$f^1_w=\pm f^1_{-w}$.
Moreover, since the intersection numbers are equal to 1, and
$f^1_w=\pm f^1_{-w}$,
then from the formula~(\ref{w5}) we have that  
$f^2_{-w}-z_{w}\in F_{w}$.
 
We denote the vector 
$f^2_{-w}-z_w$ 
by 
$\delta_w$.

Step by step, changing the gluings on the edges, we construct an orthogonal 
graph-manifold
$N$,
whose universal cover 
$\widetilde{N}$ 
is bi-Lipschitz 
equivalent to the universal cover
$\widetilde{M}$ 
of the graph-manifold 
$M$.
 
We fix an edge 
$w\in \W$.
Let it connect the vertices
$v$ 
and  
$u$.
The new gluing 
$\hg'_{-w}$ 
of blocks  
$M_v$ 
and  
$M_u$ 
we define with basis 
$z'_w=z_w+\delta_w$,
$f'_w=f_w$.
  
Thus, the isomorphism
$\hg'_{-w}$ 
is obtained from the isomorphism
$\hg_{-w}$ 
by the translation by the vector
$\delta_w$ 
on the first coordinate.
That is, such a gluing combines vectors
$f^1_w$ 
and  
$f^1_{-w}$,
as well as vectors
$z_{w}$ 
and  
$f^2_{-w}$.
 
Since with such a modification of the gluing the lattices
$P_v$,
$v\in \V$ 
and  
$F_e$,
$e\in 
\E$ 
(see Definitions~\ref{def:ind}, and \ref{def:in2}) 
do not change, then
neither the intersection number nor the secondary intersection number change.

Cutting the graph-manifold
$M$ 
along the torus
$T_{|w|}$,
and then gluing along it with the gluing
$\hg'_{-w}$,
we obtain the graph-manifold  
$N$.

\begin{lemma}\label{lem:perestr}
The universal covers of graph-manifold
$M$ 
and  
$N$ 
are bi-Lipschitz homeomorphic.
\end{lemma}
\begin{proof}
Graph-manifolds  
$M$ 
and  
$N$ 
have a common graph  
$G$,
and hence, 
the Bass-Serre tree of the graph-manifold 
$M$ 
coincides with the  same tree of the graph-manifold 
$N$.
 
Moreover, for each vertex
$v'\in \V$ 
blocks  
$M_{v'}$ 
and 
$N_{v'}$ 
are isomorphic. 

The universal cover  
$\widetilde{M}$ 
of the graph-manifold 
$M$ 
is divided into blocks, dual to the Bass-Serre tree 
$T_M$,
each of which is the universal cover of some block 
of the graph-manifold
$M$.

We call the blocks that cover the block
$M_v$ 
{\it distinguished} 
blocks.

The universal cover  
$\widetilde{N}$ 
of the graph-manifold
$N$ 
also is divided into blocks. 
Since Bass-Serre trees of these graph-manifolds coincide, 
then blocks of the manifold
$\widetilde{N}$ 
are copies of
blocks of the manifold  
$\widetilde{M}$.

The manifold  
$\widetilde{M}$ 
differs from the manifold 
$\widetilde{N}$
only by gluings along boundary components of distinguished blocks. 
We call blocks that correspond to the distinguished blocks, distinguished blocks 
too.

Now we prove the universal covers 
$\widetilde{M}$ 
and  
$\widetilde{N}$ 
are bi-Lipschitz homeomorphic. 
We construct a map
$\widetilde{M}\to \widetilde{N}$ 
in the following way:
each non-distinguished block of the manifold 
$\widetilde{M}$ 
we map identically to the corresponding 
non-distinguished block of the manifold 
$\widetilde{N}$.
For each distinguished block 
$\widetilde{M}_v=\widetilde{\Phi_v}\times \R^2$ 
our map induces a map of the boundary of these block 
to the boundary of the corresponding distinguished block 
$\widetilde{N}_v=\widetilde{\Phi_v}\times \R^2$. 
This map is identical on  each boundary component 
that does not correspond to the edge
$w$. On the boundary component 
$\ell_w\times \R^2$ that corresponds to the edge 
$w$, this map is an affine map
$A_w\colon \ell_w\times \R^2\to \ell_w\times \R^2$,
that corresponds to the map  
$h_w=(g'_{-w})^{-1}\circ g_{-w}\colon 
H_1(T_{|w|};\Z)\to H_1(T_{|w|};\Z)$.
The map  
$A_w$ 
is determined up to an integer shift in
second factor.

We expand the vector 
$\delta_w$ 
in the basis 
$(f^1_w,f^2_w)$ 
of the space 
$F_w$,
$\delta_w=\gamma_1 f^1_w+\gamma_2 f^2_w$.
 
The map
$h_w$ 
is given in the basis 
$(z_w,f^1_w,f^2_w)$ 
by formulas  
$h_w(z_w)=z_w-\delta_w=z_w-\gamma_1 f^1_w- \gamma_2 f^2_w$,
$h_w(f^1_w)=f^1_w$,
$h_w(f^2_w)=f^2_w$.
  
Consider a coordinate system  
$(x,y,z)$ 
on the boundary component
$\ell_w\times \R^2$,
so that the line 
$y=z=0$ 
corresponds to the direction
$z_w$, 
the line 
$x=z=0$ 
corresponds to the direction
$f^1_w$,
and the line  
$x=y=0$
corresponds to the direction
$f^2_w$.

In this coordinate system, the map 
$h_w$ 
corresponds to the class
$\mathcal{A}_w$ 
of maps 
$A\colon \ell_w\times \R^2\to \ell_w\times \R^2$ 
looking like  
$$
A(x,y,z)=(x,y-\g_1\cdot x + c_1, z-\g_2\cdot x + c_2),\quad (c_1,c_2)\in \Z^2.
$$

Consider a function 
$\phi_1\colon \d\widetilde{\Phi_v}\to \R$ 
which equals to 
$-\gamma_1$
on the components that correspond to the edge 
$w$ 
and equals to 
$0$ 
on other components. 
This function defines a closed 1-form on the boundary of the compact surface 
$\Phi_v$
with boundary.

Similarly a function 
$\phi_2\colon \d\widetilde{\Phi_v}\to \R$ 
that equals to $-\gamma_2$,
on the components that correspond to the edge 
$w$ 
and equals to 
$0$ 
on other components,  
defines a closed 1-form on the boundary of the compact surface 
$\Phi_v$
with boundary.

From Lemma~\ref{lemm:KL} we have that there exists a smooth Lipschitz function  
$h_1\colon 
\widetilde{\Phi_v}\to \R$ 
satisfying 
$d h_1|_{\d \widetilde{\Phi_v}}=\phi_1$.
 
Similarly, there exist a smooth Lipschitz function
$h_2\colon \widetilde{\Phi_v}\to 
\R$,
satisfying 
$d h_2|_{\d \widetilde{\Phi_v}}=\phi_2$.

In other words, the restrictions of functions
$h_1$ 
and  
$h_2$ 
on the boundary components of the surface  
$\widetilde{\Phi}_v$ 
are affine functions. 

By construction, 
on each boundary component 
$\sigma$ 
of the block  
$M_v$ 
the homeomorphism  
$\widehat{h}\colon \widetilde{M}_v\to \widetilde{N}_v$,
given by formula 
$\widehat{h}(x,y,z)=(x,y+h_1(x),z+h_2(x))$,
differs 
from some map from the class
$\mathcal{A}_w$ 
on a bounded vector
$(c^\sigma_1,c^\sigma_2)\in \R^2$.
We consider the Lipschitz functions 
$h'_1,h'_2\colon \widetilde{\Phi_v}\to \R$ 
with the support in a sufficiently small neighborhood of the boundary
$\d \widetilde{\Phi}_v$,
for which on each boundary component
$\sigma$ 
of the block 
$M_v$ 
we have 
$h_1=c^\sigma_1$ 
and 
$h_2=c^\sigma_2$.

Then the difference
$h(x,y,z)=\widehat{h}(x,y,z)-(0,h'_1(x),h'_2(x))$ 
is a required bi-Lipschitz homeomorphism.
\end{proof}

For the graph-manifold
$N$,
for the edge  
$-w$,
opposite
to the edge 
$w$,
of the graph  
$G$,
similarly to the one described above,
we construct a gluing 
$\hg'_w\colon L_{-w}\to L_{w}$,
which combines vectors
$f^1_w$ 
and   
$f^1_{-w}$,
and combines vectors  
$z_{-w}$ 
and 
$f^2_{w}$.
 
This gluing does not change the vectors
$z_{w}$ 
and  
$f^2_{-w}$.
 
Cutting the graph-manifold
$N$ 
along the torus
$T_{|w|}$,
and then gluing along it with the gluing
$\hg'_{w}$,
we obtain a graph-manifold  
$N'$,
whose gluing along the edge 
$|w|$ 
is an orthogonal.

It follows from Lemma~\ref{lem:perestr} 
that the universal covers of 
graph-manifolds 
$M$ 
and  
$N'$ 
are bi-Lipschitz homeomorphic. 

Applying the above operation sequentially to all opposing pairs
of edges 
$(w,-w)$ 
of the graph  
$G$,
we obtain from the graph-manifold  
$M$ 
a graph-manifold  
$N$. 
The universal cover of the graph-manifold $N$  
is bi-Lipschitz homeomorphic to the universal cover of the graph-manifold
$M$.
 
This proves Theorem~\ref{thm:type2}.\qed

\section{A criterion of orthogonality}\label{sec:exmpl}

In this section, we present a criterion of orthogonality for 4-dimensional
graph-manifolds for which the type of each vertex is equal to~2. 
As a consequence, we construct an example of a 4-dimensional graph-manifold that 
is not
orthogonal, all blocks of which have type 2, and intersection numbers, and 
secondary 
intersection numbers are equal to~1. 

We recall the definition of the charge map of a graph-manifold from~\cite{BK} 
(for the case of graph-manifolds  
$M$ 
of arbitrary dimension, 
$\dim M=n$).

Below, we turn to the homology groups with real coefficients, keeping
the same notation.  
In particular, 
we define 
$F_w\otimes_\Z\R$,
by $F_w$ 
and $L_{|w|}\otimes_\Z\R$.
we define by
$L_{|w|}$. 

For each directed edge  
$w$ 
of the graph  
$G$ 
of the graph-manifold  
$M$,
the gluing matrix  
$g_w$
defines a map 
$D_w:F_{-w}\to F_w$, 
such that
$D_w(f_{-w}p_w)=f_wd_wp_w$,
where 
$p_w\in \R^{n-2}$~---
is a column of reals.  
In other words, the map 
$D_w$
is defined in bases  
$f_{-w}$, and 
$f_w$
by the sub-matrix
$d_w$
of matrix 
$g_w$. 
This map can be interpreted as a projection of the space
$F_{-w}$ 
onto the space 
$F_w$ 
along the vector 
$z_w$.

In particular, the map 
$D_w$
is the identity map at the intersection 
$F_{-w}\cap F_w$.

For each directed edge
$w$ 
of the graph 
$G$ 
of the graph-manifold 
$M$,
we fix an orientation of the space  
$F_w$.

Let  
$u_w=f^1_w\wedge f^2_w\wedge\ldots\wedge f^{n-2}_w$.
Identifying spaces
$L_{-w}$ 
and  
$L_w$ 
with the map  
$g_w$,
we obtain a space 
$L_{|w|}$ 
with the couple of oriented subspaces
$F_w$ 
and 
$F_{-w}$.

Under these conditions, 
we have the canonical intersection orientation 
$u_{w\cap {-w}}$ 
on the subspace $F_w\cap F_{-w}$ 
(see~\cite{BK}).

\begin{definition}
\textit{The charge map} of the vertex 
$v\in V$ 
is a restriction 
$$K_v\colon Q_v\to F_v$$
of the map 
$$
\bigoplus\limits_{w\in \d v} \frac{1}{i_w}D_w\colon 
\bigoplus\limits_{w\in \d v} F_{-w}\to F_v
$$
onto subspace  
$Q_v$,
where 
$Q_v\subset \bigoplus\limits_{w\in \d v} F_{-w}$ 
consists of all vectors  
$q_v=\bigoplus\limits_{w\in \d v}q_{-w}$,
such that there exists a number 
$\alpha\in \R$,
that 
$q_{-w}\wedge u_{w\cap -w} =\alpha\cdot u_{-w}$ 
for each  
$w\in\d v$.
\end{definition}

This subspace does not depend on the choice of the Waldhausen basis
$(z,f)$
and for its dimension we have 
$\dim Q_v=(n-3)|\d v|+1$
(for details see~\cite{BK}). 
Note that the subspace
$$A_v=\{q_v=\bigoplus\limits_{w\in \d v}q_{-w}\mid q_{-w}\wedge u_{w\cap -w}
   =0\}\subset Q_v$$ 
is a hyperspace in
$Q_v$,
$\dim A_v=(n-3)|\d v|$.

In the 3-dimensional case,
$n=3$,
the map  
$K_v\colon Q_v\to F_v$ 
is a linear map of 1-dimensional spaces,
and therefore it is uniquely determined by a rational number
$k_v$,
the charge of the vertex 
$v$.

Although the charge map in higher dimensions is not
a number, nevertheless, 
we can say about the vanishing of the vertex charge.

\begin{definition}
{\em The charge of a vertex
$v\in V$ is vanishing}, 
iff the kernel 
$K_v$
of the charge map 
is not contained in the subspace
$A_v\sub Q_v$,
$\ker K_v\not\sub A_v$.

In this case we write 
$k_v=0$.
\end{definition}
  
\begin{remark}
Since in the 3-dimensional case
$\dim F_v=\dim Q_v=1$,
$A_v=\{0\}$
and, consequently, the condition 
$k_v=0$ 
is equivalent to 
$\ker K_v=Q_v$,
i.e  
$\ker K_v\not\sub A_v$, 
then our definition coincides with 
the regular definition of 
$k_v=0$.
\end{remark}

Let 
$M$
be a 4-dimensional graph-manifold,  all blocks of which have type~2. 
For each vertex  
$v$ 
of the graph 
$G$ 
of the manifold 
$M$, 
and 
for each edge 
$w$ 
from  
$v$ 
to 
$u$,
there are two intersection lattices in the  fiber lattice 
$F_u$ 
of the block  
$M_u$. 
We denote one of them, which is not an intersection lattice for the edge
$w$ 
by 
$\bar{P}_u$.
 
Also denote  
$\bar{P}_u\otimes_{\Z} \R$ 
by  
$J_{-w}$.

\begin{definition} 
We define {\it the subspace of intersection vectors} in the space
$Q_v$ 
by the next formula   
$$
B_v:=Q_v\cap \bigoplus\limits_{w\in \d v} J_{-w}. 
$$
\end{definition}

\begin{remark}
It follows from the definition that the subspace of intersection vectors for the 
vertex
$v$ 
does not depend on the Waldhausen basis
and, consequently, is a topological invariant of the graph-manifold 
$M$.
\end{remark}

\begin{lemma}\label{lem:ortcharge}
For any orthogonal graph-manifold 
$M$ 
and for any vertex  
$v$ 
of its graph  
$G$ 
we have
$B_v\subset \ker{K_v}$.
 
Moreover, the charge of any vertex of  
$M$ 
is vanishing.
\end{lemma}

\begin{proof}
According to Remark~\ref{rem:ort}, 
we can choose on each blocks of the graph-manifolds $M$ 
the Waldhausen basis so that for each edge 
$w\in \W$,  
bases  
$(z_w,f_w)$ 
and 
$(z_{-w},f_{-w})$ 
of the groups   
$L_w$ 
and 
$L_{-w}$,
differ only by a permutation of the elements, and, perhaps,
the signs before vectors.
Fix a vertex  
$v$.

By the orthogonality, for each edge  
$w\in \d v$,
the subspace  
$J_{-w}$ 
is generated by the vector
$z_w$.
Therefore, 
$B_v\subset \ker{K_v}$.
Moreover, 
we can choose such a sign  
$\e_w=\pm 1$ 
before the vector 
$z_{w}$,
that  
$\e_w\cdot z_{w}\wedge u_{w\cap -w} =1\cdot 
u_{-w}$.

In this way, 
$q_v=\bigoplus\limits_{w\in \d v}\e_w\cdot z_{w}\in Q_v$,  
and,  at the same time, 
$q\notin A_v$.

We have that   
$K_v(q_v)=0$, 
and consequently
$k_v=0$.
\end{proof}

\begin{remark}
It follows from Lemma~\ref{lemm:ind1}, and Definition~\ref{def:in2} that, 
intersection numbers of any edge, and secondary intersection number of any 
vertex 
of the orthogonal graph-manifold is equal to~1. 
\end{remark}

This means that the obstructions to orthogonality of graph-manifolds are
the difference between an intersection number or a secondary intersection number 
from~1.
However, even in the class of graph-manifolds, all blocks of which have type~2,
intersection numbers are equal to~1, and secondary intersection numbers are 
equal to~1, 
there exists non-orthogonal graph-manifolds. 
Below (see Corollary~\ref{cor:notort}) we give an example of such a
graph-manifolds.

\begin{theorem}\label{thm:crit}
Let  
$M$
be a graph-manifold, all blocks of which have type~2.
The graph-manifold  
$M$ 
is orthogonal, iff the following three conditions are satisfied 
\begin{itemize}
 \item[(1)] the intersection number of each edge is equal to 1;~\label{i1} 
 \item[(2)] the secondary intersection number of each vertex is equal to 
1;~\label{i2} 
 \item[(3)] the subspace of the intersection vectors of each vertex is contained 
in
             the kernel of the charge map 
             $B_v\subset \ker{K_v}$.~\label{i3} 
\end{itemize}
\end{theorem}

\begin{proof}
If  
$M$ 
is orthogonal, then from Lemma~\ref{lemm:ind1}, 
Definition~\ref{def:in2}, and Lemma~\ref{lem:ortcharge} we have that 
conditions~(1)--(3) 
are satisfied. 

In the opposite direction. Since the type of any block is equal to~1.
It follows from the equality of the secondary intersection numbers to~1 
that we can choose such a Waldhausen basis 
$\{(z_w,f_w)\mid w\in \d v, v\in \V\}$,
that for every block 
$M_v$ 
we have 
$f^1_v\in P^1_v$ 
and 
$f^2_v\in P^2_v$. 

For each edge  
$w\in \W$, 
from the vertex  
$v$ 
to the vertex  
$u$, 
the gluing 
$\hg_{-w}$ 
of the blocks  
$M_v$ 
and  
$M_u$ 
is given by bases 
$(z_w,f_w)$ 
and 
$(z_{-w},f_{-w})$ 
of the space  
$L_{|w|}$.

In other words, the matrix
$g_{-w}$ 
of this gluing is obtained by expanding of the basis
$(z_w,f_w)$ 
in the basis  
$(z_{-w},f_{-w})$.
We can assume that
$P^1_u=P_w=P^1_v$,
where  
$P_w=F_w\cap F_{-w}$.
Then, in this notation,  
$f^1_w=\pm f^1_{-w}$.
Moreover, since the intersection numbers are equal to~1, and
$f^1_w=\pm f^1_{-w}$, 
then from the formula~(\ref{w5}) we have that  
$f^2_{-w}-z_{w}\in F_{w}$.
 
It follows from the condition~(3) that,  
$q=\bigoplus\limits_{w\in \d v} f^2_{-w}\in 
\ker{K_v}$.
It means that  
$\sum\limits_{w\in \d v} D_w(f^2_{-w})=0$.
On the other hand, since  
$f^2_{-w}-z_{w}\in F_{w}$,
then  
$D_w(f^2_{-w})=f^2_{-w}-z_{w}$.
 
Denote  
$f^2_{-w}-z_{w}$ 
by 
$n_w$.
From the properties~\ref{w3}, and \ref{w4}, 
and the equality  
$\sum\limits_{w\in \d v}{n_w}=0$ 
we have that there exists such a Waldhausen basis 
$\{(\bar{z}_w,\bar{f}_w)\mid w\in \d v, v\in 
\V\}$ that  
$\bar{z}_w=z_q+n_w$ 
and 
$\bar{f}_w=f_w$ 
for any directed edge of the graph  
$G$.
 
For such a basis the following conditions are satisfied
$\bar{f}^1_w=\bar{f}^1_{-w}$,
$\bar{f}^2_{-w}=\bar{z}_w$ 
and  
$\bar{f}^2_{w}=\bar{z}_{-w}$.
 
It means that the manifold  
$M$ 
is orthogonal. 
\end{proof}

\begin{corollary}\label{cor:notort}
There exists a 4-dimensional non-orthogonal graph-manifold,  
all blocks of which have type~2, all intersection numbers are equal to~1, 
and all secondary intersection numbers are equal to~1. 
\end{corollary}

\begin{proof}
As a graph
$G$
of a graph-manifold $M$ 
we take a cycle of the length
$k\geq 3$.

We number its vertices  
$v_1$,
\ldots, 
$v_k$.
 
For each vertex  
$v_i$,
$i=1,\ldots, k$,
we consider a block 
$M_i=\Phi\times 
T^2$,
where  
$\Phi$ is a torus with 2 boundary components. 
We glue the graph-manifold from blocks so that for each edge 
$w\in \W$ 
the corresponding gluing matrix is equal to 
$$g_w=g_w^{z,f}=\left(\begin{array}{ccc}
  0 & 0 & 1\\
  0 & 1 & 0\\
  1 & 0 & 0\\
     \end{array}\right).
$$
It follows from the definition (see subsection~\ref{subsec:grm}) that 
the resulting graph-manifold is orthogonal. 
Consequently, by Theorem~\ref{thm:crit} for each vertex  
$v\in \V$ 
the set of the intersection vectors belongs to the kernel of 
the charge map,  
$B_v\subset \ker{K_v}$.

Consider the edge
$w$ 
from the vertex  
$v_2$ 
to the vertex 
$v_3$.
 
Replacing the gluing on this edge by the gluing
$$\bar{g}_w=\bar{g}_w^{z,f}=\left(\begin{array}{ccc}
  0 & 0 & 1\\
  0 & 1 & 1\\
  1 & 0 & 0\\
     \end{array}\right), 
$$
we obtain a graph-manifold 
$M'$, 
all blocks of which have type~2
It follows from Lemma~\ref{lemm:ind1}, and Definition~\ref{def:in2} that, 
all intersection numbers of 
$M'$ 
are equal to~1, 
and all secondary intersection numbers of 
$M'$
are equal to~1.
On the other hand, the charge map of the vertex
$v_1$ 
does not change.
At the same time, the spaces of the intersection vectors are different.
Indeed, consider edges 
$w_k$ 
and 
$w_2$ 
from 
$v_1$
to 
$v_k$ 
and
$v_2$ 
respectively. 
Then a new space of intersection vectors 
$B'_{v_1}$ 
is obtained from the old one by a
translation on the vector 
$0+f^1_{-w_2}\in F_{-w_k}\oplus F_{-w_2}$). 
It means that 
$K_v(q)=f^1_{v_1}$,
where  
$q=f^2_{-w_k}\oplus f^2_{-w_2}\neq 
0$.
 
That is,  
$B'_{v_1}$ 
does not belong to the kernel 
$\ker{K_{v_1}}$. 
It follows from Theorem~\ref{thm:crit},  that 
the graph-manifold 
$M'$ 
is not orthogonal. 
\end{proof}


\begin{thebibliography}{9}

\bibitem{W1} \textsc{Waldhausen F.}:\ \textit{Eine Klasse von 3-dimensionalen 
Mannigfaltigkeiten. 
I}, Invent. Math. 3 (1967), 308--333.

\bibitem{W2} \textsc{Waldhausen F.}:\ \textit{Eine Klasse von 3-dimensionalen 
Mannigfaltigkeiten. 
II}, Invent. Math. 4 (1967), 87--117.
 
\bibitem {BK} \textsc{Buyalo S. V., Kobel'skii V. L.}:\ 
\textit{Generalized graph-manifolds of nonpositive curvature}, 
St. Petersburg Math. J.,
\textbf{Volume 11, Issue 2} (2000), 251--268.

\bibitem {BS} \textsc{Buyalo S., Schroeder V.}:\ \textit{Elements of asymptotic 
geometry},  European Mathematical Society, 2007.

\bibitem {Gro} \textsc{Gromov M.}:\ \textit{Geometric group theory Vol. 2: 
Asymptotic invariants of infinite groups.}, 1996.

\bibitem {HS}  \textsc{Hume D., Sisto A}:\ \textit{Embedding universal covers 
of 
graph manifolds in products of trees}, Proc. Amer. Math. Soc. 141 (2013), no. 
10, 3337--3340.

\bibitem {KL}  \textsc{Kapovich M., and Lieb B.}:\ \textit{3-manifold groups, 
and 
nonpositive curvature}, 
Geometric Analysis, and Functional Analysis,
\textbf{Volume 8} (1998), p. 841--852.

\bibitem {Sc}  \textsc{Scott P.}:\  \textit{The Geometries of 3-Manifolds}, 
Bulletin of the London Mathematical Society, 15(5), 401--487.

\bibitem {Smir1}  \textsc{Smirnov A. V.}:\ \textit{Quasi-isometric embedding of 
the fundamental 
group of an orthogonal graph-manifold into a product of metric trees},
St. Petersburg Math. J., \textbf{Volume 24:5} (2013), 811--821. 

\end{thebibliography}
\end{document}